\documentclass[12pt]{article}
\usepackage[utf8]{inputenc}
\usepackage{amsthm}
\usepackage{amsfonts}
\usepackage{amsmath}
\usepackage{colonequals}
\usepackage{txfonts}
\usepackage{url}
\usepackage{asymptote}
\usepackage{graphicx}
\newtheorem{theorem}{Theorem}

\newtheorem{conjecture}{Conjecture}
\newtheorem{lemma}[theorem]{Lemma}

\title{Triangle-free planar graphs with at most $64^{n^{0.731}}$ 3-colorings}

\author{Zdeněk Dvořák\thanks{Charles University, Prague, Czech Republic.  E-mail: {\tt rakdver@iuuk.mff.cuni.cz}.  Supported in part by ERC Synergy grant DYNASNET no. 810115.}\and
Luke Postle\thanks{University of Waterloo. E-mail: {\tt lpostle@uwaterloo.ca}. Partially supported by NSERC under Discovery Grant No. 2019-04304 and the Canada Research Chairs program.}}
\date{}

\begin{document}
\maketitle

\begin{abstract}
Thomassen conjectured that triangle-free planar graphs have exponentially many 3-colorings.
Recently, he disproved his conjecture by providing examples of such graphs with $n$ vertices
and at most $2^{15n/\log_2 n}$ 3-colorings.  We improve his construction, giving examples of
such graphs with at most $64^{n^{log_{9/2} 3}}<64^{n^{0.731}}$ 3-colorings.  We conjecture this exponent is optimal.
\end{abstract}

There are many instances in graph theory where if an $n$-vertex graph from some class is guaranteed to contain some
structure (perfect matching, coloring, \ldots), it is actually guaranteed to have at least $c^n$ of them for some fixed constant $c>1$.
For example, every $n$-vertex planar graph has at least $60\cdot 2^{n-3}$ proper 5-colorings~\cite{exp5},
and every $n$-vertex planar graph of girth at least five has at least $2^{n/9}$ proper 3-colorings~\cite{thom-many}.

By Grötzsch' theorem, every triangle-free planar graph is 3-colorable, and Thomassen conjectured that
such graphs actually have exponentially many 3-colorings.  To provide some support for this conjecture,
Thomassen~\cite{thom-many} proved that every $n$-vertex triangle-free planar graph has at least $c^{n^{1/12}}$ proper 3-colorings
for $c=2^{1/20000}$.  This bound was improved by Asadi, Dvořák, Postle, and Thomas~\cite{submany} to $2^{\sqrt{n/212}}$.
However, Thomassen~\cite{nonexp2021} recently disproved his conjecture, showing that for infinitely many $n$,
there exists an $n$-vertex triangle-free planar graph with at most $2^{15n/\log_2 n}$ 3-colorings.
In this note, we further improve Thomassen's construction.

\begin{theorem}\label{thm-main}
For infinitely many integers $n$, there exists an $n$-vertex triangle-free planar graph with at most
$64^{n^{log_{9/2} 3}}<64^{n^{0.731}}$ 3-colorings.
\end{theorem}

As it does not seem possible to further improve the exponent using the same method, we conjecture this bound
is tight.

\begin{conjecture}
There exists a constant $c>1$ such that every $n$-vertex triangle-free planar graph has at least $c^{n^{log_{9/2} 3}}$
distinct proper 3-colorings.
\end{conjecture}

\section{The construction}

For a positive integer $b$, let $P(u,v,b)$ denote the graph consisting of vertices $u$ and $v$ and a path $v_1v_2\ldots v_b$,
with $u$ adjacent to $v_1$, $v_3$, \ldots and $v$ adjacent to $v_2$, $v_4$, \ldots; see Figure~\ref{fig-constr}(a).

\begin{figure}
\begin{center}
\includegraphics{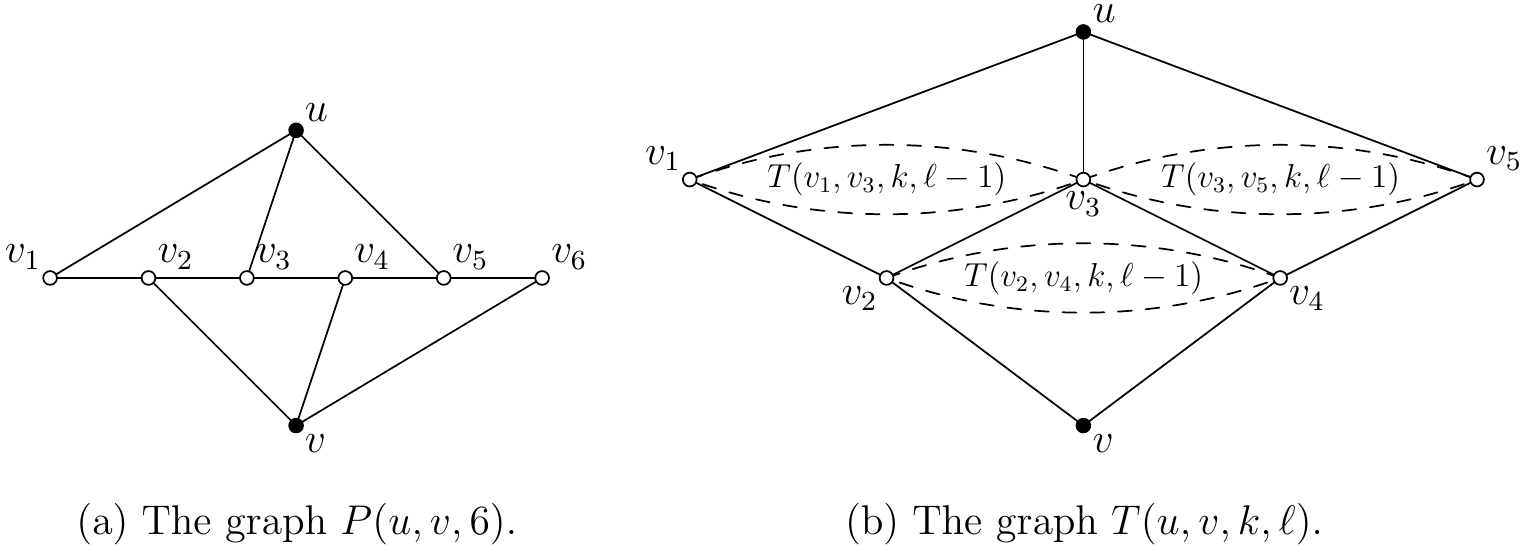}
\end{center}
\caption{The graphs used in the construction.\label{fig-constr}}
\end{figure}

\begin{lemma}\label{lemma-cnt}
Consider the graph $P=P(u,v,5)$, and let $\psi$ be any $3$-coloring of $P$.
\begin{itemize}
\item[(a)] At least one of $\psi(v_1)=\psi(v_3)$, $\psi(v_2)=\psi(v_4)$, and $\psi(v_3)=\psi(v_5)$ holds.
\item[(b)] If $\psi(u)=\psi(v)$, then $\psi(v_1)=\psi(v_3)=\psi(v_5)$ and $\psi(v_2)=\psi(v_4)$.
\end{itemize}
\end{lemma}
\begin{proof}
Without loss of generality, we can assume $\psi(u)=3$.
For the claim (a), if $\psi(v_1)\neq \psi(v_3)\neq\psi(v_5)$, then since $v_1$, $v_3$, and $v_5$ are adjacent to $u$, we can
without loss of generality assume $\psi(v_1)=\psi(v_5)=1$ and $\psi(v_3)=2$.  However, then $\psi(v_2)=\psi(v_4)=3$.
The claim (b) clearly holds, since if $\psi(v)=3$, then the path $v_1v_2\ldots v_5$ must be colored alternately by colors $1$ and $2$.
\end{proof}
Furthermore, note that for any $b\ge 1$, the graph $P(u,v,b)$ has exactly two 3-colorings that assign both $u$ and $v$ the color $1$.

For any integers $\ell\ge 0$ and $k\ge 1$, let us now define a triangle-free plane graph $T(u,v,k,\ell)$ with two special non-adjacent 
vertices $u$ and $v$ incident with its outer face.
\begin{itemize}
\item For $\ell=0$, we have $T(u,v,k,0)=P(u,v,2^k)$.
\item For $\ell>0$, the graph $T(u,v,k,\ell)$ is obtained from $P(u,v,5)$ by adding
$T(v_1,v_3,k,\ell-1)$, $T(v_2,v_4,k,\ell-1)$, and $T(v_3,v_5,k,\ell-1)$ in the 4-faces, see Figure~\ref{fig-constr}(b).
\end{itemize}

Note that $T(u,v,k,\ell)$ contains $3^\ell$ copies of the graph $P(x,y,2^k)$. Let us denote these copies by $P_1$, \ldots, $P_{3^\ell}$,
where $P_i=P(x_i,y_i,2^k)$; and let us call $(x_1,y_1)$, \ldots, $(x_{3^\ell}, y_{3^\ell})$ the \emph{leaf pairs}.
Let us remark that this implies the number of vertices of $T(u,v,k,\ell)$ is at least
\begin{equation}\label{eq-nall}
3^\ell 2^k.
\end{equation}
Let $V_\ell$ denote the set of vertices of $T(u,v,k,\ell)$ not belonging to $V(P_i)\setminus\{x_i,y_i\}$
for any $i$; note that this set is independent of $k$.  We have $|V_0|=2$ and $|V_{\ell}|=3|V_{\ell-1}|+1$ for $\ell>0$,
and thus
\begin{equation}\label{eq-inner}
|V_\ell|=3^{\ell}|V_0|+\sum_{i=0}^{\ell-1}3^i < 2.5\cdot 3^{\ell}.
\end{equation}

\begin{lemma}\label{lemma-cext}
Suppose $\ell>0$ and consider any 3-coloring $\psi$ of the subgraph of $T(u,v,k,\ell)$ induced by $V_\ell$.
Then $\psi$ extends to at most $2^{2^{k+\ell}+3^\ell}$ 3-colorings of $T(u,v,k,\ell)$.
\end{lemma}
\begin{proof}
By Lemma~\ref{lemma-cnt}(a), there exists $i\in\{1,2,3\}$ such that $\psi(v_i)=\psi(v_{i+2})$.  By iterated
application of Lemma~\ref{lemma-cnt}(b), we conclude that $\psi(x_j)=\psi(y_j)$ for every leaf pair $(x_j,y_j)$
contained in the 4-cycle $v_iv_{i+1}v_{i+2}z$, where $z=v$ if $i=2$ and $z=u$ otherwise.
Repeating the same argument in the subgraphs $T(v_l,v_{l+2},k,\ell-1)$ for $l\in\{1,2,3\}\setminus \{i\}$,
we conclude that $\psi(x_j)=\psi(y_j)$ holds for all but at most $(2/3)^\ell\cdot 3^\ell=2^\ell$ leaf pairs.
Since $\psi$ extends to a 3-coloring of $P_j$ in 2 ways if $\psi(x_j)=\psi(y_j)$ and in at most $2^{2^k}$ ways
if $\psi(x_j)\neq\psi(y_j)$, the number of extensions of $\psi$ to a 3-coloring of $T(u,v,k,\ell)$
is at most
$$\bigl(2^{2^k}\bigr)^{2^\ell}\cdot 2^{3^\ell}=2^{2^{k+\ell}+3^\ell}.$$
\end{proof}

Since the subgraph of $T(u,v,k,\ell)$ induced by $V_\ell$ can be 3-colored in at most
$3\cdot 2^{|V_\ell|-1}\le 2^{2.5\cdot 3^{\ell}+1}$ ways (see~(\ref{eq-inner})), we conclude that the number of
$3$-colorings of $T(u,v,k,\ell)$ is at most
\begin{equation}\label{eq-totcol}
2^{2.5\cdot 3^{\ell}+1}\cdot 2^{2^{k+\ell}+3^\ell}<2^{2^{k+\ell}+4\cdot 3^\ell}.
\end{equation}

Let us choose $k=\lceil \ell\cdot \log_2 (3/2)\rceil$,
so that $3^\ell\le 2^{k+\ell}\le 2\cdot 3^\ell$, and let $n_\ell$ and $c_\ell$ denote the number of vertices and 3-colorings of $T(u,v,k,\ell)$
for this choice of $k$.  Note that by (\ref{eq-nall}),
$$n_\ell\ge 3^\ell 2^k\ge (9/2)^\ell$$
and by (\ref{eq-totcol}),
$$c_\ell\le 2^{2^{k+\ell}+4\cdot 3^\ell}\le 2^{6\cdot 3^\ell}=2^{6\cdot n_{\ell}^{\log_{9/2} 3}}.$$
This finishes the proof of Theorem~\ref{thm-main}.

\bibliographystyle{siam}
\bibliography{../data.bib}
\end{document}